\documentclass[12pt,leqno]{amsart}
\usepackage{amssymb,amsthm,amsmath,latexsym}
\newtheorem{theorem}{\sc Theorem}[section]
\newtheorem{lemma}[theorem]{\sc Lemma}
\newtheorem{proposition}[theorem]{\sc Proposition}

\newtheorem{remark}[theorem]{\sc Remark}
\newtheorem*{thmA}{Theorem A}
\newtheorem*{thmB}{Theorem B}

\date{}

\title{On residually finite groups with Engel-like conditions}
\author{Raimundo Bastos}
\address{Departamento de Matem\'atica, Universidade de Brasilia, Bras\'ilia-DF, 70910-900 Brazil }
\email{bastos@mat.unb.br}
\thanks{Supported by CAPES-Brazil}

\begin{document}
\begin{abstract} Let $m,n$ be positive integers. Suppose that $G$ is a residually finite group in which for every element $x \in G$ there exists a positive integer $q=q(x) \leqslant m$ such that $x^q$ is $n$-Engel. We show that $G$ is locally virtually nilpotent. Further, let $w$ be a multilinear commutator and $G$ a residually finite group in which for every product of at most $896$ $w$-values $x$ there exists a positive integer $q=q(x)$ dividing $m$ such that $x^q$ is $n$-Engel. Then $w(G)$ is locally virtually nilpotent. 
\end{abstract}

\maketitle

\section{Introduction}

Let $G$ be a group. Let $x,y \in G$ we then define $[x,{}_1y]=[x,y]$ and $[x,{}_{i+1}y]=[[x,{}_{i}y],y]$ for $i\geq1$, where. The group $G$ is called an {\em Engel group} if for every $x,y\in G$ there is a positive integer $n=n(x,y)$ such that $[x,{}_{n}y]=1$. An element $y\in G$ is called {\em (left) Engel} if for any $x\in G$ there is a positive integer $n=n(x)$ such that $[x,{}_{n}y]=1$. Similarly, we say that $y$ is {\em (left) $n$-Engel} if for any $x\in G$ we have $[x,{}_{n}y]=1$. The group $G$ is called a {\em $n$-Engel group} if $[x,{}_{n}y]=1$ for all $x,y \in G$.    

According to the solution of the Restricted Burnside Problem (Zelmanov, \cite{ze1,ze2}) every residually finite group of finite exponent is locally finite. Another interesting result in this context was due to Wilson \cite{w} which states that every $n$-Engel residually finite  group is locally nilpotent. 

In this work, we study residually finite groups in which some powers are $n$-Engel. In a certain way, our results can be viewed as generalizations of the above results.

\begin{thmA} \label{a} Let $m,n$ be positive integers. Suppose that $G$ is a residually finite group in which for every element $x \in G$ there exists a positive integer $q=q(x) \leqslant m$ such that $x^q$ is $n$-Engel. Then $G$ is locally virtually nilpotent.
\end{thmA} 

In Theorem A it is essential the hypothesis of the boundedness of the powers of each element. It is well known that there are periodic residually finite groups which are not locally finite. Some examples with this property have been constructed in \cite{Aleshin,Gri,G,GS,S}. In particular, these groups cannot be locally virtually nilpotent.  

Recall that if $w$ is a group-word and $G$ is a group, then the verbal subgroup $w(G)$ of $G$ corresponding to the group-word $w$ is the subgroup generated by all $w$-values in $G$. A group-word $w$ is a {\it multilinear commutator} if it can be written as a multilinear Lie monomial. An important family of multilinear commutators consists of the lower central words $\gamma_k$, given by 
\[
\gamma_1=x_1,
\qquad
\gamma_k=[\gamma_{k-1},x_k]=[x_1,\ldots,x_k],
\quad
\text{for $k\ge 2$.}
\]

The corresponding verbal subgroups $\gamma_k(G)$ are the terms of the lower central series of $G$. When $k=2$ we write $G'$ rather than $\gamma_2(G)$. 

Another distinguished sequence of multilinear commutator words are the derived words $\delta_k$, on $2^k$ variables, which are defined recursively by
\[
\delta_0=x_1,
\quad
\delta_k=[\delta_{k-1}(x_1,\ldots,x_{2^{k-1}}),\delta_{k-1}(x_{2^{k-1}+1},\ldots,x_{2^k})].
\]

The verbal subgroup that corresponds to the word $\delta_k$ is the familiar $k$-th derived subgroup of $G$ usually denoted by $G^{(k)}$.

In \cite{BSTT} was proved that given positive integers $m,n$ and a multilinear commutator word $v$, if $G$ is a residually finite group in which all values of the word $w=v^m$ are $n$-Engel, then $w(G)$ is locally nilpotent. For more details concerning groups with $n$-Engel word-values see \cite{BS,BSTT,shu1,STT}. In \cite{shu3} Shumyatsky has showed that if $w$ is a multilinear commutator and $G$ is a residually finite group in which every product of at most $896$ $w$-values has order dividing $n$, then $w(G)$ is locally finite. In the present paper we establish the following related result.

\begin{thmB}\label{b} Let $m,n$ be positive integers and $w$ a multilinear commutator. Suppose that $G$ is a residually finite group in which for any product of at most $896$ $w$-values $x$ there exists a positive integer $q=q(x)$ dividing $m$ such that $x^q$ is $n$-Engel. Then $w(G)$ is locally virtually nilpotent.
\end{thmB} 

The above theorem is no longer valid if the assumption of residual finiteness of $G$ is dropped. In \cite{DK} Deryabina and Kozhevnikov showed that for any integer $k \geqslant 1$ there exists an integer $N = N(k) \geqslant 1$ such that for every odd number $n > N$ there is a group $G$ with commutator subgroup $G'$ is not locally finite and satisfying the identity $$f = ([x_1,x_2] \ldots [x_{2k-1},x_{2k}])^n \equiv~1.$$ In particular, $G'$ cannot be locally virtually nilpotent.

\section{Preliminaries}

We say that a set $X$ is  {\it commutator-closed} if $[x,y]\in X$ whenever $x,y\in X$.

\begin{lemma} \label{soluble}
Let $X$ be a normal commutator-closed subset of a group $G$. Assume that $G$ is generated by finitely many elements of $X$. If $G$ is soluble, then $G$ is finite.  
\end{lemma}

\begin{proof}
Let $G$ be a counter-example whose derived length is as small as possible. As the quotient $G/G'$ is an abelian group generated by finitely many elements of finite order, we have $G'$ is a finitely generated subgroup. Since $X$ is a normal commutator-closed subset of $G$, we conclude that $G'$ is also generated by finitely many elements of $X$. From this we see that $G'$ is finite, which completes the proof.  
\end{proof}

Recall that a group is {\it locally graded} if every non-trivial finitely generated subgroup has a proper subgroup of finite index. Interesting classes of groups (e.g. locally finite groups, locally nilpotent groups, residually finite groups) are locally graded.

It is clear that a quotient of a residually finite group need not be residually finite (see for instance \cite[6.19]{Rob}). In particular, this occurs also with locally graded groups. However, the next result gives a sufficient condition for a quotient to be locally graded.

\begin{lemma} (Longobardi, Maj, Smith, \cite{LMS}) \label{HP}
Let $G$ be a locally graded group and $N$ a normal locally nilpotent subgroup of $G$. Then $G/N$ is locally graded.
\end{lemma}

Shumyatsky in \cite{shu3} has showed that if $w$ is a multilinear commutator and $G$ is a residually finite group in which for any product of at most $896$ $w$-values $x$ there exists a positive integer $q=q(x)$ dividing a fixed positive integer $m$ such that $x^q = 1$, then the verbal subgroup $w(G)$ is locally finite. Next, we extend this result to the class of locally graded groups. 

\begin{lemma}\label{grad} Let $k,m,n$ be positive integers. Suppose that $G$ is a locally graded group in which for every product of at most $896$ $\delta_k$-values $x$ there exists a positive integer $q=q(x)$ dividing $m$ such that $x^q = 1$. Then $G^{(k)}$ is locally finite.
\end{lemma} 

\begin{proof} 
Denote by $X$ the set of all $\delta_k$-values in $G$. Choose arbitrarily a finitely generated subgroup $V$ of $G^{(k)}$. There are finitely many $\delta_k$-values $h_1,\dots,h_s$ such that $$V \leqslant H = \langle h_1,\dots,h_s \rangle.$$ It suffices to prove that $H$ is finite. Let $R$ be the finite residual of $H$, i.e., the intersection of all subgroups of finite index in $H$. If $R=1$, then $H$ is residually finite. By \cite[Theorem 4.1]{shu3}, $H^{(k)}$ is locally finite. According to Lemma \ref{soluble} the quotient $H/H^{(k)}$ is finite and so $H^{(k)}$ is finitely generated. Consequently, $H$ is finite. So, we can assume that $R\neq1$. Since $H/R$ is residually finite, similarly we obtain that $H/R$ is finite. Hence $R$ is finitely generated. As $R$ is locally graded we have $R$ contains a proper subgroup of finite index in $H$, which gives a contradiction. 
\end{proof}

We need the following result, due to Shumyatsky \cite{shu3}. 

\begin{lemma} \label{soluble.w}
Let $w$ be a multilinear commutator, $G$ a soluble group in which all $w$-values have finite order. Then the verbal subgroup $w(G)$ is locally finite.
\end{lemma}

This next result will be needed in the proof of Theorem B. 

\begin{proposition}\label{graded} Let $m$ be a positive integer and $w$ a multilinear commutator. Suppose that $G$ is a locally graded group in which for every product of at most $896$ $w$-values $x$ there exists a positive integer $q=q(x)$ dividing $m$ such that $x^q = 1$. Then $w(G)$ is locally finite.
\end{proposition} 

\begin{proof} By Lemma 4.2 of \cite{shu3} we have a positive integer $k$ such that every $\delta_k$-value in $G$ is a $w$-value. Therefore $G^{(k)}$ is locally finite by Lemma \ref{grad}. Further, the group $G$ satisfying the identity $$f = w^m \equiv~1.$$ In particular, every $w$-value in $G$ has finite order. It follows that the quotient $w(G)/G^{(k)}$ is locally finite (Lemma \ref{soluble.w}). Thus $w(G)$ is locally finite, as required.
\end{proof}

\section{Associated Lie algebras}

Let $L$ be a Lie algebra over a field $\mathbb{K}$.
We use the left normed notation: thus if
$l_1,l_2,\dots,l_n$ are elements of $L$, then
$$[l_1,l_2,\dots,l_n]=[\dots[[l_1,l_2],l_3],\dots,l_n].$$
We recall that an element $a\in L$ is called {\it ad-nilpotent} if there exists a positive integer $n$ such that $[x,{}_na]=0$ for all $x\in L$. When $n$ is the least integer with the above property then we say that $a$ is ad-nilpotent of index $n$. 

Let $X\subseteq L$ be any subset of $L$. By a commutator of elements in $X$,
we mean any element of $L$ that could be obtained from
elements of $X$ by means of repeated operation of
commutation with an arbitrary system of brackets
including the elements of $X$. Denote by $F$ the free
Lie algebra over $\mathbb{K}$ on countably many free
generators $x_1,x_2,\dots$. Let $f=f(x_1,x_2,
\dots,x_n)$ be a non-zero element of $F$. The algebra
$L$ is said to satisfy the identity $f=0$ if
$f(l_1,l_2,\dots,l_n)=0$ for any $l_1,l_2,\dots,l_n
\in L$. In this case we say that $L$ is PI. Now, we recall an important theorem of Zelmanov 
\cite[Theorem 3]{zelm} that has had many applications to the Group Theory.
\begin {theorem}\label{1} 
Let $L$ be a Lie algebra generated by $a_1,a_2,\dots,a_m$. Assume that $L$ is PI and that each commutator in the generators is ad-nilpotent. Then $L$ is nilpotent.
\end{theorem}

\begin{center}
 {\it On Lie Algebras Associated with Groups}
\end{center}

Let $G$ be a group and $p$ a prime. Let us denote by $D_i=D_i(G)$ 
the $i$-th dimension subgroup of $G$ in characteristic
$p$. These subgroups form a central series of $G$
known as the {\it Zassenhaus-Jennings-Lazard series} (see \cite[Section 2]{shu2} for more details). Set $L(G)=\bigoplus D_i/D_{i+1}$. 
Then $L(G)$ can naturally be viewed as a Lie algebra 
over the field ${\mathbb F}_p$ with $p$ elements. 

The subalgebra of $L$ generated by $D_1/D_2$ will be 
denoted by $L_p(G)$. The nilpotency of $L_p(G)$ has strong influence in the structure of a finitely generated group $G$. The following result is due to Lazard \cite{l2}.

\begin{theorem}\label{3} 
Let $G$ be a finitely generated pro-$p$
group. If $L_p(G)$ is nilpotent, then $G$ is
$p$-adic analytic.
\end{theorem}

Let $x\in G$, and let $i=i(x)$ be the largest integer such
that $x\in D_i$. We denote by ${\tilde x}$ the 
element $xD_{i+1}\in L(G)$. Now, we can state one condition for ${\tilde x}$ to be ad-nilpotent.

\begin{lemma} \label{4}(Lazard, \cite[page 131]{la}) 
For any $x\in G$ we have
$(ad\,{\tilde x})^q=ad\,(\widetilde {x^q})$.
\end{lemma}

\begin{remark}
We note that $q$ in Lemma \ref{4} does not need to be a 
$p$-power. In fact, is easy to see that if $p^s$ is the 
maximal $p$-power dividing $q$, then $\tilde x$ is 
ad-nilpotent of index at most $p^s$.
\end{remark}

The following result is an immediate corollary of \cite[Theorem 1]{wize}.
\begin{lemma}\label{2}
Let $G$ be any group satisfying a group-law. Then $L(G)$ is $PI$.
\end{lemma}

\section{Proofs of the main results}

We denote by $\mathcal{N}$ the class of all finite nilpotent groups. The following result is a straightforward corollary of \cite[Lemma 2.1]{w}.

\begin{lemma}\label{Wilson}
Let $G$ be a finitely generated residually-$\mathcal{N}$ group. For each prime $p$, let $R_p$ be the intersection of all normal subgroups of $G$ of finite $p$-power index. If $G/R_p$ is nilpotent for each $p$, then $G$ is nilpotent.
\end{lemma}

In any group $G$ there exists a unique maximal normal locally nilpotent
subgroup (called the Hirsch-Plotkin radical) containing all normal locally nilpotent subgroups of $G$ \cite[12.1.3]{Rob}. We denote by $HP(G)$ the Hirsch-Plotkin radical of the group $G$. It is well know that for any group G, all elements in $HP(G)$ are Engel. But, it is also know that the set of all Engel elements of the group $G$ can differ from $HP(G)$ (cf. \cite{G}). According to Gruenberg \cite[12.3.3]{Rob} the Hirsch-Plotkin radical of a soluble group is precisely the set of all Engel elements. The next result is taken from \cite{BSTT}.

\begin{lemma}\label{asce} 
Let $G$ be a group with an ascending normal series whose factors are locally soluble. Then the set of all Engel elements of $G$ coincides with $HP(G)$.
\end{lemma}

\begin{proposition} \label{prop}
Let $k,m,n$ be positive integers. Suppose that $G$ is a residually-$\mathcal{N}$ group in which for every $\delta_k$-value $x$ there exists a positive integer $q=q(x) \leqslant m$ such that $x^q$ is $n$-Engel.  Then the Hirsch-Plotkin radical $HP(G)$ is precisely the set of Engel elements of $G$. 
\end{proposition}

\begin{proof}
By Lemma \ref{asce}, it is sufficient to show that finitely generated subgroups of $\delta_k(G)$ is locally soluble. 

Choose arbitrary $\delta_k$-value $x \in G$ and $q=q(x)$ a positive integer such that $q \leqslant m$ and the element $x^q$ is $n$-Engel. We will prove that the normal closure of $ x^q$ in $G$, $\langle (x^q)^h \ \vert \ h \in G  \rangle$, is soluble. Let $b_1, \ldots,b_t$ be finitely many elements in $G$. Let $h_i = (x^q)^{b_i}$, $i=1, \ldots,t$ and $H = \langle h_1, \ldots, h_t\rangle$. We will prove that $H$ is soluble. As a consequence of Lemma \ref{Wilson}, we can assume that $G$ is residually-$p$ for some prime $p$. Let $L=L_p(H)$ be the Lie algebra associated with the Zassenhaus-Jennings-Lazard series $$H=D_1\geq D_2\geq \cdots$$ of $H$. Then $L$ is generated by $\tilde{h}_i=h_i D_2$, $i=1,2,\dots,t$. Let $\tilde{h}$ be any Lie-commutator in $\tilde{h}_i$ and $h$ be the group-commutator in $h_i$ having the same system of brackets as $\tilde{h}$. Since for any group commutator $h$ in $h_1\dots,h_t$ there exists a positive integer $q \leqslant m$ such that $h^q$ is $n$-Engel, Lemma \ref{4} shows that any Lie commutator in $\tilde h_1\dots,\tilde h_t$ is ad-nilpotent. On the other hand, for every $\delta_k$-value $x = \delta_k(a_1, \ldots, a_{2^k})$ there exists a positive integer $q = q(x) \leqslant m$ such that $x^q$ is $n$-Engel, then $G$ satisfies the identity
$$[y, {}_{n}\delta_k(a_1, \ldots, a_{2^k}), {}_{n} \delta_k(a_1, \ldots, a_{2^k})^2, \ldots, {}_{n} \delta_k(a_1, \ldots, a_{2^k})^m]\equiv 1$$
and therefore, by Lemma \ref{2}, $L$ satisfies some non-trivial polynomial identity. Now Zelmanov's Theorem \ref{1} implies that $L$ is nilpotent. Let $\hat{H}$ denote the pro-$p$ completion of $H$. Then $L_p(\hat{H})=L$ is nilpotent and $\hat{H}$ is $p$-adic analytic group by Theorem \ref{3}. Clearly $H$ cannot have a free subgroup of rank 2 and so, by Tits' Alternative \cite{tits}, $H$ is virtually soluble. As $H$ is residually-$p$ we have $H$ is soluble. Since $h_1,\dots,h_t$ have been chosen arbitrarily, we now conclude that $\delta_k(G)$ is locally soluble, which completes the proof.
\end{proof}

The following result is a consequence of Restricted Burnside Problem (see for instance \cite[Theorem 1]{Ma}).

\begin{lemma}\label{loc.finite}
Let $G$ be a locally graded group of finite exponent. Then $G$ is locally finite.   
\end{lemma}

We are now in a position to prove Theorem A.

\begin{proof}[Proof of Theorem A] Recall that $G$ is a residually finite group in which for every element $x\in G$ there exists a number $q=q(x) \leqslant m$ such that $x^q$ is $n$-Engel. We need to show that finitely generated subgroups of $G$ are virtually nilpotent.

Choose arbitrarily finitely many elements $h_1,\dots,h_s \in G$ and set $H=\langle h_1,\dots,h_d\rangle$. Let $K$ be the subgroup of $H$ generated by all Engel elements in $G$ contained in $H$. Now, it is sufficient to prove that $H/K$ is finite and $K$ is nilpotent. Since $K$ is generated by Engel elements in $G$, it follows that $K$ is locally nilpotent (Proposition \ref{prop}). According to Lemma \ref{HP} the quotient $H/K$ is locally graded. As $H/K$ is a locally graded of finite exponent we have $H/K$ is finite (Lemma \ref{loc.finite}). In particular, $K$ is finitely generated, which completes the proof. 
\end{proof}

The proof of Theorem B is now easy. 

\begin{proof} [Proof of Theorem B] Recall that $G$ is a residually finite group in which for any product of at most $896$ $w$-values $x$ there exists a positive integer $q=q(x)$ dividing $m$ such that $x^q$ is $n$-Engel. We need to prove that finitely generated subgroups of $w(G)$ are virtually nilpotent.

Let $W$ be a finitely generated subgroup of $w(G)$. Clearly, there exist finitely many $w$-values $a_1,\dots,a_s$ such that $W\leq\langle a_1,\dots,a_s\rangle$. Set $H=\langle a_1,\dots,a_s \rangle$ and $K$ be the subgroup generated by all Engel elements in $G$ contained in $H$. Arguing as in the proof of Theorem A we deduce that $K$ is locally nilpotent and $H/K$ is locally graded. It suffices to prove that the quotient $H/K$ is finite. Since $x^m \in K$ for every $w$-value $x$, it follows that the quotient $w(G)/HP(w(G))$ is locally finite (Proposition \ref{graded}). From this we deduce that $H/K$ is finite. In particular, $K$ is finitely generated. The theorem follows.    
\end{proof}

\end{document}